\documentclass[11pt,reqno]{amsart}
\usepackage{amsmath,amssymb,booktabs,subfig,times,url}

\usepackage{vmargin}
\setpapersize{USletter}
\setmargrb{1.5in}{.75in}{1.5in}{.75in}

\usepackage{graphicx}
\usepackage{tikz}
\usetikzlibrary{calc}
\usetikzlibrary{arrows}
\usetikzlibrary{patterns}
\usetikzlibrary{through}
\usetikzlibrary{plotmarks}

\newtheorem{theorem}{Theorem}[section]
\newtheorem{lemma}[theorem]{Lemma}
\newtheorem{proposition}[theorem]{Proposition}
\newtheorem{corollary}[theorem]{Corollary}
\newtheorem{conjecture}[theorem]{Conjecture}
\newtheorem*{definition}{Definition}

\newcommand{\R}{{\mathbb R}}

\title[Neumann eigenvalue sums on triangles]{Neumann eigenvalue sums on triangles are (mostly) minimal for equilaterals}
\author{R. S. Laugesen, Z. C. Pan and S. S. Son}
\address{Department of Mathematics, University of Illinois, Urbana,
IL 61801, U.S.A.} \email{Laugesen\@@illinois.edu,pan22\@@illinois.edu,son9\@@illinois.edu}
\date{\today}



\begin{document}

\maketitle

\begin{abstract}
We prove that among all triangles of given diameter, the equilateral triangle minimizes the sum of the first $n$ eigenvalues of the Neumann Laplacian, when $n \geq 3$.

The result fails for $n=2$, because the second eigenvalue is known to be minimal for the degenerate acute isosceles triangle (rather than for the equilateral) while the first eigenvalue is $0$ for every triangle. We show the third eigenvalue is minimal for the equilateral triangle.
\end{abstract}

\section{Results}
\label{results_sec}

Eigenfunctions of the Neumann Laplacian satisfy $-\Delta u = \mu u$ with natural boundary condition $\frac{\partial u}{\partial n} = 0$, and the eigenvalues $\mu_j$ satisfy
\[
 0 = \mu_1 < \mu_2 \le \mu_3 \le \dots \to \infty.
\]
We prove a geometrically sharp lower bound on sums of Neumann eigenvalues on triangular domains, under normalization of the diameter.
\begin{theorem} \label{main}
Among all triangular domains of given diameter, the equilateral triangle minimizes the sum of the first $n$ eigenvalues of the Neumann Laplacian, when $n \geq 3$.

That is, if $T$ is a triangular domain, $E$ is equilateral, and $n \geq 3$, then \[
(\mu_2 + \dots + \mu_n)D^2 \big|_T \geq (\mu_2 + \dots + \mu_n)D^2 \big|_E
\]
with equality if and only if T is equilateral.
\end{theorem}
Multiplying the eigenvalues by $D^2$ renders them scale invariant. Note the eigenvalues of the equilateral triangle are known explicitly (see Appendix~\ref{equilateral}), so that the lower bound in the theorem is computable.

We prove the theorem in Sections~\ref{sec3}--\ref{sec5}. The proof is fully rigorous except when $n = 4,5,7,8,9$. For those values of $n$, the proof relies on numerical estimation of the eigenvalues $\mu_2,\ldots,\mu_9$ for one specific isosceles triangle. See Proposition~\ref{exceptional} and Table~\ref{numerical values}, below.

Theorem~\ref{main} is \emph{geometrically} sharp, meaning there exists an extremal domain for each $n$. It is the first sharp lower bound on Neumann eigenvalue sums. (Upper bounds are due to Laugesen and Siudeja \cite{LS11a}, under a moment of inertia normalization.) The theorem differs from the Weyl-type bounds of Kr\"{o}ger \cite{K94}, which are \emph{asymptotically} sharp as $n \to \infty$, for each domain.

Two reasons for studying such sums are that the sum represents the energy needed to fill the lowest $n$ quantum states under the Pauli exclusion principle, and that the eigenvalue sum provides a ``summability'' approach to studying the high eigenvalues ($\mu_n$ for large $n$), which are difficult to study directly.

We concentrate on triangular domains because they are the simplest domains whose eigenvalues cannot be computed explicitly. The ``hot spots'' conjecture of Jeffrey Rauch \cite{K85} about the maximum of the Neumann eigenfunction $u_2$ remains unsolved on acute triangles, in spite of Ba\~{n}uelos and Burdzy's proof for obtuse triangles by coupled Brownian motion \cite{BB99}. The triangular spectral gap conjecture of Antunes and Freitas \cite{AF08}, which claims that the difference of the first two Dirichlet eigenvalues is minimal for the equilateral, also remains unsolved. (The gap minimizer among convex domains is a degenerate rectangle \cite{AC10}, but that result sheds no light on the conjecture for triangles.) Clearly much remains to be discovered about triangles!

\smallskip
Theorem~\ref{main} fails for the second eigenvalue, $n=2$, because $\mu_2 D^2$ is minimized not by the equilateral but by the degenerate acute isosceles triangle, as Laugesen and Siudeja showed when finding the optimal Poincar\'{e} inequality on triangles \cite{LS10}.

For the third eigenvalue we do prove minimality of the equilateral, in Section~\ref{misc_proof}:
\begin{corollary}\label{mu3}
Among all triangles of given diameter, $\mu_3$ is minimal for the equilateral triangle. That is, $\mu_3 D^2 \geq 16\pi^2/9$ for all triangular domains, with equality if and only if the triangle is equilateral.
\end{corollary}
The fourth eigenvalue is \emph{not} minimal for the equilateral, as shown by the numerical work in Figure~\ref{fig:tones}. The minimum appears to occur at the intersection of two eigenvalue branches.
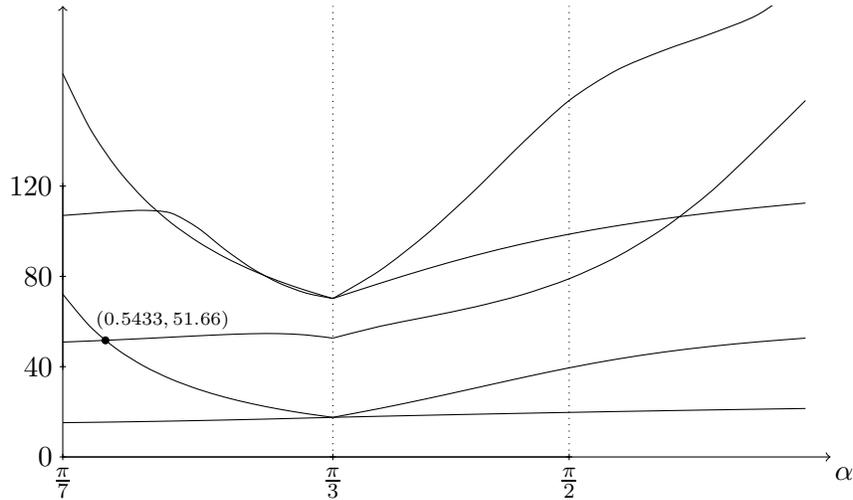
\begin{figure}[t]
  \begin{center}
    \begin{tikzpicture}[scale=6,smooth]
      \draw[<->] (pi/7,1) |- (2.15,0);
      \clip (0.3,1) rectangle (2.2,-0.1);
      \draw (pi/7,0) node[below] {$\frac{\pi}{7}$};
      \draw (pi/3,0) node[below] {$\frac{\pi}{3}$};
      \draw (pi/2,0) node[below] {$\frac{\pi}{2}$};
      \draw (2.18,0) node[below] {$\alpha$};
      \draw (pi/7,0) node[left] {$0$};
      \draw (pi/7,40*0.005) node[left] {$40$};
      \draw (pi/7,80*0.005) node[left] {$80$};
      \draw (pi/7,120*0.005) node[left] {$120$};
      \draw[dotted] (pi/3,0) -- (pi/3,1);
      \draw[dotted] (pi/2,0) -- (pi/2,1);
      \draw [mark=+, mark size=0.2] plot coordinates{(pi/7,0) (pi/3,0) (pi/2,0)};
      \draw [mark=+, mark size=0.2] plot coordinates{(pi/7,40*0.005) (pi/7,80*0.005) (pi/7,120*0.005)};

        \draw [yscale=0.005] plot coordinates{
( 0.4488,72.1139)( 0.5086,57.9203)( 0.5685,47.7810)( 0.6283,40.2687)( 0.6882,34.5373)( 0.7480,30.0580)( 0.8078,26.4863)( 0.8677,23.5897)( 0.9275,21.2061)( 0.9874,19.2199)( 1.0472,17.5467)
       };
        \draw [yscale=0.005] plot coordinates{
( 0.4488,169.7954)( 0.5086,145.6839)(0.5685,127.9234)( 0.6283,114.3962)( 0.6882,103.8134)( 0.7480,95.3488)( 0.8078,88.4487)( 0.8677,82.7279)( 0.9275,77.9090)( 0.9874,73.7858)( 1.0472,70.2009)
       };
        \draw [yscale=0.005] plot coordinates{
( 0.4488,15.1856)( 0.5086,15.3319)( 0.5685,15.4980)( 0.6283,15.6843)( 0.6882,15.8911)( 0.7480,16.1184)( 0.8078,16.3663)( 0.8677,16.6343)( 0.9275,16.9216)( 0.9874,17.2267)( 1.0472,17.5469)
         };
        \draw [yscale=0.005] plot coordinates{
( 0.4488,50.8980)( 0.5086,51.3716)( 0.5685,51.8980)( 0.6283,52.4685)( 0.6882,53.0678)( 0.7480,53.6669)( 0.8078,54.2120)( 0.8677,54.6010)( 0.9275,54.6468)( 0.9874,54.0664)( 1.0472,52.6476)
      };
         \draw [yscale=0.005] plot coordinates{
( 0.4488,106.9827)( 0.5086,107.8767)( 0.5685,108.7367)( 0.6283,109.2524)( 0.6882,108.0722)( 0.7480,101.4950)( 0.8078,91.9613)( 0.8677,83.5811)( 0.9275,77.0899)( 0.9874,72.6266)( 1.0472,70.1981)
       };
        \draw [yscale=0.005] plot coordinates{
( 1.0472,17.5467)( 1.1519,18.0249)( 1.2566,18.4824)( 1.3614,18.9202)( 1.4661,19.3391)( 1.5708,19.7398)( 1.6755,20.1222)( 1.7802,20.4866)( 1.8850,20.8324)( 1.9897,21.1593)( 2.0944,21.4665)
       };
        \draw [yscale=0.005] plot coordinates{
( 1.0472,70.2009)( 1.1519,76.9673)( 1.2566,83.3902)( 1.3614,89.2344)( 1.4661,94.3500)( 1.5708,98.7127)( 1.6755,102.3958)( 1.7802,105.5133)( 1.8850,108.1772)( 1.9897,110.4795)( 2.0944,112.4895)
         };
        \draw [yscale=0.005] plot coordinates{
( 1.0472,17.5469)( 1.1519,21.5130)( 1.2566,25.8608)( 1.3614,30.4675)( 1.4661,35.1089)( 1.5708,39.4804)( 1.6755,43.3055)( 1.7802,46.4620)( 1.8850,48.9914)( 1.9897,51.0112)( 2.0944,52.6407)
         };
        \draw [yscale=0.005] plot coordinates{
( 1.0472,52.6476)( 1.1519,57.9005)( 1.2566,62.0794)( 1.3614,66.4473)( 1.4661,71.8275)( 1.5708,78.9687)( 1.6755,88.5963)( 1.7802,101.3303)( 1.8850,117.6049)( 1.9897,137.2928)( 2.0944,157.9429)
      };
         \draw [yscale=0.005] plot coordinates{
( 1.0472,70.1981)( 1.1519,82.6791)( 1.2566,98.9861)( 1.3614,118.1248)( 1.4661,138.8719)( 1.5708,157.9454)( 1.6755,171.2929)( 1.7802,180.1369)( 1.8850,187.5220)( 1.9897,196.0473)( 2.0944,210.5943)
       };
      \fill (0.543282,51.6623*0.005) circle (0.25pt) ;
	  \draw (0.67,52*0.005) node [above] {\tiny $(0.5433,51.66)$} ;
    \end{tikzpicture}
  \label{figl}

  \end{center}
  \caption{Numerical plot of the first five nonzero Neumann eigenvalues normalized by diameter ($\mu_j D^2$ for $j=2,\ldots,6$) of an isosceles triangle, computed by the PDE Toolbox in Matlab and plotted as a function of the aperture angle $\alpha$ between the two equal sides. The minimum value of $\mu_4 D^2$ is approximately $51.66$, occurring at $\alpha \simeq 0.5433$ (to $4$ significant figures). The value at the equilateral triangle ($\alpha=\pi/3$) is larger: $\mu_4 D^2=3\cdot 16\pi^2/9 \simeq 52.64$.}
  \label{fig:tones}
\end{figure}

Now let us consider other shapes. Among rectangles of a given diameter, the square does not always minimize the sum of the first $n$ Neumann eigenvalues. For example, by plotting the first $12$ eigenvalues as a function of side-ratio, one finds that the square fails to minimize $(\mu_2+\dots+\mu_n)D^2$ when $n=5,6,7,10,11,12$.


Ellipses behave more agreeably, for each individual eigenvalue, as we prove in Section~\ref{misc_proof}:
\begin{proposition}\label{ellipse}
Among ellipses of given diameter, the disk minimizes each eigenvalue of the Neumann Laplacian. That is, for each $j \geq 2$, the quantity $\mu_j D^2$ is strictly minimal when the ellipse is a disk .
\end{proposition}

What about general convex domains? Our result for triangles in Theorem~\ref{main}, together with  Proposition~\ref{ellipse} for ellipses, suggests that:
\begin{conjecture}
Among convex domains of given diameter, the disk minimizes the sum of the first $n \geq 3$ eigenvalues of the Neumann Laplacian. That is,
$(\mu_2+\dots+\mu_n)D^2$ is minimal when the domain is a disk, for each $n \geq 3$.
\end{conjecture}
The conjecture fails for $n=2$, because Payne and Weinberger proved $\mu_2 D^2$ is minimal for the degenerate rectangle (and not the disk) among all convex domains \cite{PW60}. In other words, they proved that the optimal Poincar\'{e} inequality for convex domains is saturated by the degenerate rectangle.

\subsection*{Dirichlet and Robin boundary conditions}
Minimality of \emph{Dirichlet} eigenvalue sums for the equilateral, among all triangles of given diameter, was proved recently by Laugesen and Siudeja \cite{LS11}, for each $n \geq 1$. We will adapt their Method of the Unknown Trial Function to the Neumann case. The adaptation breaks down for triangles that are ``close to equilateral'' when $n=4,5,7,8,9$, as we see in the next section. To overcome that obstacle we introduce a new triangle with which to compare, in Proposition~\ref{exceptional}. The eigenvalues of this triangle are not known explicitly, which necessitates a numerical evaluation for those exceptional $n$- values.

Similar results should presumably hold under Robin boundary conditions, although no such results have been proved. The Method of the Unknown Trial Function seems not to work there, because the boundary integral in the Robin Rayleigh quotient transforms differently from the integrals over the domain, under linear maps.

For more information on isoperimetric-type eigenvalue inequalities in mathematical physics (the general area of this paper), see the survey by Ashbaugh \cite{A99}, and the monographs of Bandle \cite{B80}, Henrot \cite{He06}, Kawohl \cite{K85}, Kesavan \cite{K06} and P\'{o}lya--Szeg\H{o} \cite{PS51}.

\section{Method of the Unknown Trial Function: the proof of Theorem~\ref{main}}
\label{sec3}

\begin{definition}\rm
The \emph{aperture} of an isosceles triangle is the angle between
its two equal sides. Call a triangle \emph{subequilateral} if it is isosceles with
aperture less than $\pi/3$, and \emph{superequilateral} if it is
isosceles with aperture greater than $\pi/3$.
\end{definition}

The theorem will be proved in three steps.

\medskip
\noindent Step 1 --- Reduction to subequilateral triangles. Suppose the given triangle is not equilateral. We may suppose it is subequilateral, as follows. Stretch the triangle in the direction perpendicular to its longest side, until one of the other two sides has the same length as the longest one. This subequilateral triangle has the same diameter as the original triangle, and has strictly smaller eigenvalue sums by Lemma~\ref{stretching} later in the paper. (When applying the equality statement of that lemma, notice that a second-or-higher Neumann eigenfunction of a triangle cannot depend only on $x$, because the boundary condition would force such a function to be constant.)

Thus it suffices to prove the theorem for subequilateral triangles.

\medskip
\noindent Step 2 --- Method of the Unknown Trial Function.
Write
\[
M_n = \mu_2 + \cdots + \mu_n
\]
for the sum of the first $n$ eigenvalues (where we omit $\mu_1=0$). Define
\[
T(a,b) = \text{triangle having vertices at $(-1,0), (1,0)$ and $(a,b)$,}
\]
where $a \in \R$ and $b>0$. The triangle $T(a,b)$ is isosceles if $a=0$, and subequilateral if in addition $b>\sqrt{3}$. We will prove the theorem for the subequilateral triangle $T(0,b)$ with $b>\sqrt{3}$.

Further define three special triangles
\begin{align*}
E & = T(0,\sqrt{3}) = \text{equilateral triangle,} \\
F_+ & = T(+1,2\sqrt{3})  = \text{30-60-90 right triangle,} \\
F_- & = T(-1,2\sqrt{3})  = \text{30-60-90 right triangle.}
\end{align*}
The spectra of these triangles are explicitly computable, as we shall need in Step 3 below. Notice $F_+$ and $F_-$ have the same spectra, by symmetry.

Our method involves transplanting the ``unknown' eigenfunctions of the triangle $T(0,b)$ to obtain trial functions for the (known) eigenvalues of the special triangles $E,F_+,F_-$; see Figure~\ref{fig:lineartrans}. By this technique we will prove:

\begin{figure}[t]
  \begin{center}
	\begin{tikzpicture}[scale=1.5]
	  \draw[very thick,red] (-1,0) -- (1,0) -- (0,2.5) node [right] {$(0,b)$} -- cycle;
	  \draw[thick, blue] (-1,0) -- (1,0) -- (0,{sqrt(3)}) -- cycle;
	  \draw (0,1/2) node [below] {$E$} ;
	  \draw (-1,0) node [below] {$(-1,0)$} -- (1,0) -- (-1,{2*sqrt(3)}) -- cycle;
	  \draw (-1,2) node [right] {$F_-$} ;
	  \draw (1,0) node [below] {$(1,0)$} -- (-1,0) -- (1,{2*sqrt(3)}) -- cycle;
	  \draw[green!50!black,dashed,thick] (0,2.5) -- (-1,{2*sqrt(3)}) ;
	  \draw (1,2) node [left] {$F_+$} ;
	  \draw[green!50!black,dashed,thick] (0,2.5) -- (1,{2*sqrt(3)}) ;
	  \draw[green!50!black,dashed,thick] (0,2.5) -- (0,{sqrt(3)}) ;
	\end{tikzpicture}
  \end{center}
  \caption{Linear maps to the subequilateral triangle $T(0,b)$, from the equilateral triangle $E$ and right triangles $F_+,F_-$.}
  \label{fig:lineartrans}
\end{figure}
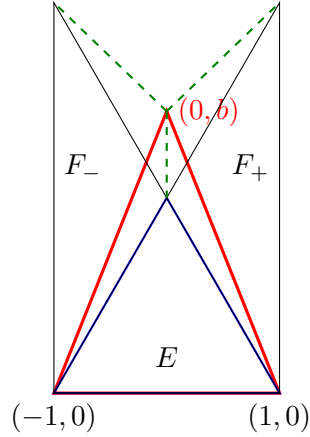

\begin{proposition} \label{mutf}
For each $n \geq 2$:
\begin{enumerate}
\item[(a)] if $b>\sqrt{3}$ then
\[
M_n D^2 \big|_{T(0,b)} > \min\{M_n D^2 \big|_E, \frac{6}{11} M_n D^2 \big|_{F_\pm} \} ;
\]
\item[(b)] if $b \geq 2.14$, then a better lower bound holds, namely
\[
M_n D^2 \big|_{T(0,b)} > \min\{M_n D^2 \big|_E, \frac{5}{8} M_n D^2 \big|_{F_\pm}\} .
\]
\end{enumerate}

\end{proposition}
The proof is in Section~\ref{sec4}.

\medskip
\noindent Step 3 --- Compare eigenvalues of right and equilateral triangles.
\begin{lemma} \label{comrt} \

\begin{enumerate}
\item[(a)]
$\frac{6}{11}M_n D^2\big|_{F_\pm} \geq M_n D^2\big|_E$ for $n=3,6$ and each $n \geq10$.
\item[(b)]
$\frac{5}{8}M_n D^2\big|_{F_\pm} \geq M_n D^2\big|_E$ for $n=4,5,7,8,9$.
\end{enumerate}
\end{lemma}
The lemma is proved in Section~\ref{sec5}. The lemma is certainly plausible, because the Weyl asymptotic ($\mu_j \sim 4\pi j/A$ as $j \to \infty$) implies that $M_n D^2$ is about twice as large for the half-equilateral $F_\pm$ as for the full equilateral $E$, when $n$ is large.

Proposition~\ref{mutf} combines with Lemma~\ref{comrt} to show $M_n D^2 \big|_{T(0,b)} > M_n D^2 \big|_E$ for most cases needed in Theorem~\ref{main}. The remaining cases, where $n=4,5,7,8,9$ and $\sqrt{3} < b < 2.14$, are treated in the next proposition.
\begin{proposition} \label{exceptional}
The statement
\[
M_n D^2\big|_{T(0,b)}>M_n D^2\big|_E, \qquad b \in (\sqrt{3},2.14),
\]
is true when
\begin{itemize}
\item
$n=4$ \ if $M_4 D^2\big|_G>90.73$,
\item
$n=5$ \ if $M_5 D^2\big|_G>163.31$,
\item
$n=7$ \ if $M_7 D^2\big|_G>362.90$,
\item
$n=8$ \ if $M_8 D^2\big|_G>489.91$,
\item
$n=9$ \ if $M_9 D^2\big|_G>653.22$.
\end{itemize}
Here $G$ denotes the isosceles triangle $T(0,2.14)$.
\end{proposition}
We verify the hypotheses of this proposition numerically in Table~\ref{numerical values}. The proposition itself is proved in Section~\ref{sec4}.

\begin{table}[t]
\begin{center}
 \begin{tabular}{ll}
 \toprule
   $n$     &  $M_n D^2 \big|_G$ \\
   \midrule
   4     &  \ \ 94.59      \\
   5     &  176.73     \\
   6     &  259.48     \\
   7     &  379.58     \\
   8     &  530.54     \\
   9     &  712.65     \\
   \bottomrule
 \end{tabular}
\end{center}
\medskip
\caption{Numerical values of the diameter-normalized eigenvalue sum $M_n D^2=(\mu_2 + \cdots + \mu_n)D^2$ for the isosceles triangle $G=T(0,2.14)$, computed using the PDE Toolbox in Matlab.}
\label{numerical values}
\end{table}

\section{Linear transformation of unknown eigenfunctions: proof of Propositions~\ref{mutf} and \ref{exceptional}}
\label{sec4}

Write $\mu_j(a,b)$ for the Neumann eigenvalues of the triangle $T(a,b)$, and let the $u_j$ be corresponding orthonormal eigenfunctions. Write
\[
M_n(a,b) = \mu_2(a,b) + \cdots + \mu_n(a,b)
\]
for the eigenvalue sum.

We need a lemma estimating the change in an eigenvalue sum when the triangle undergoes linear transformation.
\begin{lemma}[Linear transformation and eigenvalue sums]\label{lemtrace}
Let $a,c\in R $ and $b,d>0$. Take $C > 0$ and $n \geq 2$. Then the inequality
\[
M_n(a,b) > C M_n(c,d)
\]
holds if
\[
\frac{1}{d^2} \Big[ \big( (a-c)^2+d^2 \big) (1-\gamma_n) + 2b(a-c)\delta_n + b^2\gamma_n \Big] < \frac{1}{C},
\]
where
\[
\gamma_n = \frac{\sum_{j=1}^n \int_{T(a,b)} u_{j,y}^2 \, dA}{\sum_{j=1}^n \int_{T(a,b)} |\nabla u_j|^2 \, dA} \qquad \text{and} \qquad
\delta_n = \frac{\sum_{j=1}^n \int_{T(a,b)} u_{j,x} u_{j,y} \, dA}{\sum_{j=1}^n \int_{T(a,b)} |\nabla u_j|^2 \, dA} .
\]
\end{lemma}

\subsection*{Proof of Proposition~\ref{mutf}}
Laugesen and Siudeja proved part (a) of Proposition~\ref{mutf} in \cite[Proposition~3.1]{LS11} Their proof treated Dirichlet boundary conditions, but the Neumann argument is identical --- except here we need $n \geq 2$ in order to avoid dividing by zero in the definition of $\gamma_1$.

Now we prove part (b), by adapting Laugesen and Siudeja's proof. Assume $b>\sqrt{3}$. The equilateral triangle $E = T(0,\sqrt{3})$ has diameter $2$, and the subequilateral triangle $T(0,b)$ has diameter $\sqrt{1+b^2}$. The inequality
\[
   \left. M_n D^2 \right|_{T(0,b)} = M_n(0,b) (1+b^2) > M_n(0,\sqrt{3}) 2^2 = \left. M_n D^2 \right|_E
\]
will hold by Lemma~\ref{lemtrace} with $a=c=0, d=\sqrt{3}$ and $C=2^2/(1+b^2)$ if
\[
(1-\gamma_n) + \frac{1}{3} b^2 \gamma_n < \frac{1+b^2}{2^2} .
\]
This last inequality is equivalent to $\gamma_n < 3/4$. Thus if $\gamma_n<3/4$ then part (b) of the Proposition is proved. Assume $\gamma_n \geq 3/4$ from now on.

The triangle $F_\pm = T(\pm 1,2\sqrt{3})$ has diameter $4$. The inequality
 \[
   \left. M_n D^2 \right|_{T(0,b)} = M_n(0,b) (1+b^2) > \frac{5}{8} M_n(\pm 1,2\sqrt{3}) 4^2 = \frac{5}{8} \left. M_n D^2 \right|_{F_\pm}
\]
will hold by Lemma~\ref{lemtrace} with $a=0, c=\pm 1, d=2\sqrt{3}$ and $C=\frac{5}{8}\frac{4^2}{1+b^2}$ if
\[
 \frac{1}{12}[13(1-\gamma_n)\mp 2b\delta_n + b^2\gamma_n]<\frac{8}{5}\frac{1+b^2}{4^2}.
\]
We only need this inequality to hold for one of the choice of ``$+$'' or ``$-$'', because $F_+$ and $F_-$ have the same eigenvalues. Thus it suffices to show
\[
 \frac{1}{12}[13(1-\gamma_n)+b^2\gamma_n]<\frac{8}{5}\frac{1+b^2}{4^2},
\]
which is equivalent to
\[
b^2>13-\frac{19}{6-5\gamma_n}.
\]
The maximum of the right hand side over all possible values of $\gamma_n \in [\frac{3}{4},1]$ is approximately $(2.134)^2$. Thus part (b) certainly holds  under the assumption $b \geq 2.14$.

\subsection*{Proof of Proposition~\ref{exceptional}}
In the previous proof we compared the eigenvalue sums of the subequilateral triangle $T(0,b)$ with those of the right triangles $F_\pm$, by means of the Method of the Unknown Trial Function. Those comparisons proved insufficient when $b < 2.14$. So in this current proof we compare with the ``endpoint'' triangle $T(0,b_*)$. Unfortunately, the eigenvalues of this triangle are not explicitly computable, which explains why certain explicit estimates appear in the hypotheses of this Proposition.

We want to prove $M_n D^2 \big|_{T(0,b)} > M_nD^2 \big|_E$, for $\sqrt{3}<b<b_*$ and $n=4,5,7,8,9$, where we have defined $b_*=2.14$. The proof of Proposition~\ref{mutf} above proves this inequality when $\gamma_n < \frac{3}{4}$. So we assume $\gamma_n \geq \frac{3}{4}$.

Let $K=0.967$. We will first prove
\begin{equation} \label{newest}
M_nD^2 \big|_{T(0,b)} = M_n(0,b)(b^2+1)>KM_n(0,b_*)(b_*^2+1) = KM_nD^2 \big|_{T(0,b_*)} .
\end{equation}
This inequality holds by Lemma~\ref{lemtrace} with $a=c=0, d=b_*$ and $C=K \frac{b_*^2+1}{b^2+1}$ if
\begin{equation} \label{desiderata}
1 - \gamma_n + \frac{b^2}{b_*^2} \gamma_n < \frac{1}{K} \frac{b^2+1}{b_*^2+1} .
\end{equation}
We must show that this inequality holds for all $\gamma_n \in [\frac{3}{4},1]$ and all $b \in (\sqrt{3},b_*)$. Fixing $b$ temporarily, we see that the left side of inequality \eqref{desiderata} is maximized when $\gamma_n = \frac{3}{4}$. Substituting $\gamma_n = \frac{3}{4}$ and then rearranging, we see it suffices to prove
\[
K < \frac{4b_*^2}{3(b_*^2+1)} \frac{b^2+1}{b^2+b_*^2/3}
\]
for all $b \in (\sqrt{3},b_*)$. The right side of this new inequality is an increasing function of $b$, since $b_*^2/3>1$. Thus it suffices to check the inequality at $b=\sqrt{3}$; one finds the right side equals approximately $0.9671$, which exceeds our chosen value of $K=0.967$ on the left side. Hence \eqref{newest} is proved.

To complete the proof that $M_n D^2 \big|_{T(0,b)} > M_nD^2 \big|_E$, from \eqref{newest}, it would suffice to know
\[
M_n D^2\big|_{T(0,b_*)} > \frac{1}{K} M_n D^2\big|_E .
\]
The right hand side can be evaluated explicitly (using the eigenvalues of the equilateral triangle $E$ as calculated in the Appendix). For $n=4,5,7,8,9$ it equals $90.73$, $163.31$, $362.90$, $489.91$, $653.22$, respectively. (We have rounded each number up in the second decimal place.) These calculations justify the appearance of the five numbers in the hypotheses of the proposition.

\section{Comparison of eigenvalue sums: proof of Lemma~\ref{comrt}}\label{sec5}

Consider the eigenvalue counting function $N(\mu)=\#\{j \geq 0:\mu_j(E_1)<\mu\}$, where $E_1$ is an equilateral triangle with sidelength $1$. We develop explicit bounds of Weyl type on this counting function, and then apply the bounds to prove Lemma~\ref{comrt}.

\begin{lemma} \label{counting1}
The counting function satisfies
\[
\frac{\sqrt{3}}{16\pi}\mu+\frac{(6-\sqrt{3})}{4\pi}\sqrt{\mu}+\frac{3}{2}>N(\mu)>\frac{\sqrt{3}}{16\pi}\mu+\frac{\sqrt{3}}{4\pi}\sqrt{\mu}-\frac{3}{2},  \qquad \text{for all $\mu>48\pi^2$.}
\]
Hence for all $j \geq 26$,
\begin{align*}
& \frac{16\pi}{\sqrt{3}}(j-\frac{3}{2})-8(2\sqrt{3}-1)\sqrt{\frac{4\pi}{\sqrt{3}}(j-\frac{3}{2})+13-4\sqrt{3}}+8(13-4\sqrt{3}) \\
& \leq \mu_j(E_1) \\
& < \frac{16\pi}{\sqrt{3}}(j+\frac{1}{2})-8\sqrt{\frac{4\pi}{\sqrt{3}}(j+\frac{1}{2})+1}+8 .
\end{align*}
\end{lemma}

\begin{proof}[Proof of Lemma~\ref{counting1}]
The spectrum of the equilateral triangle $E_1$ under the Neumann Laplacian is well known (see Appendix A):
\[
\sigma_{m,n}=\frac{16 \pi^2}{9}(m^2+mn+n^2), \qquad m, n \geq 0.
\]
Hence the Neumann counting function equals
\[
N(\mu) = \# \big\{ (m,n) : m, n \geq 0, (m^2+mn+n^2) < R^2 \big\} ,
\]
where $R=3\sqrt{\mu}/4\pi$. The difference between this formula and the counting function $N_D(\cdot)$ for the Dirichlet eigenvalues is that in the Dirichlet case, $m$ and $n$ must be positive. Therefore by counting pairs $(m,n)$ that have either $m=0$ or $n=0$, we can relate the two counting functions as follows:
\[
N_D(\mu) + 2R + 1 > N(\mu) > N_D(\mu) + 2(R-1)+1 ,
\]
where the ``$+1$'' counts the pair $(0,0)$. Some known estimates on the Dirichlet counting function $N_D$ (see \cite[Lemma~5.1]{LS11}) now imply our estimates on the Neumann counting function in Lemma~\ref{counting1}.

Next, by applying the upper estimate in the lemma with $\mu=48\pi^2+1$, we find $N(48\pi^2+1)<26$. We conclude that $\mu_j \geq 48\pi^2+1$ whenever $j \geq 26$. Thus the counting function bounds in the lemma can be inverted for each $j \geq 26$ to yield the stated bounds on $\mu_j$. (Specifically, to invert the upper bound on the counting function one puts $\mu=\mu_j+\epsilon$ and uses that $N(\mu_j+\epsilon) \geq j$; to invert the lower bound one puts $\mu=\mu_j$ and uses that $j-1 \geq N(\mu_j)$.)
\end{proof}

Let $\mu^s_j(E_1)$ be the $j$th symmetric eigenvalue of the equilateral triangle $E_1$ (see Appendix~\ref{equilateral}), and write $N^s(\mu)$ for the symmetric counting function.

\begin{lemma} \label{counting2}
The symmetric counting function satisfies
\[
 N^s(\mu)<\frac{\sqrt{3}}{32\pi}\mu+\frac{3}{4\pi}\sqrt{\mu}+\frac{5}{4}, \qquad \text{for all $\mu>48\pi^2$.}
\]
Hence for all $j \geq 15$,
\[
 \mu^s_j(E_1)\ge\frac{32\pi}{\sqrt{3}}(j-\frac{5}{4})-32\sqrt{2\sqrt{3}\pi(j-\frac{5}{4})+9}+96.
\]
\end{lemma}
\begin{proof}[Proof of Lemma~\ref{counting2}]
The symmetric eigenvalues of the equilateral triangle $E_1$ are
\[
\sigma_{m,n}=\frac{16\pi^2}{9}(m^2+mn+n^2), \qquad m \geq n \geq 0 ,
\]
so that
\[
 N^s(\mu)=\# \big\{ (m,n): m \geq n \geq 0, (m^2+mn+n^2) < R^2 \big\} ,
\]
where $R=3\sqrt{\mu}/4\pi$. Hence by symmetry,
\[
2N^s(\mu) \leq N(\mu) + R/\sqrt{3} + 1,
\]
where the term ``$+R/\sqrt{3} + 1$'' estimates the number of pairs $(m,n)$ with $m=n$. Now the upper bound on $N^s(\mu)$ in the lemma follows from the upper bound on $N(\mu)$ in Lemma~\ref{counting1}.

Next, by applying the upper estimate in this lemma with $\mu=48\pi^2+1$, we find $N^s(48\pi^2+1)<15$, so that $\mu^s_j \geq 48\pi^2+1$ whenever $j \geq 15$. Thus the counting function estimate in the lemma can be inverted to yield the stated bounds on $\mu^s_j$, for each $j \geq 15$.
\end{proof}

\subsection*{Proof of Lemma~\ref{comrt}}
The right triangle $F_+=T(1,2\sqrt{3})$ is half of an equilateral triangle. Thus the Neumann eigenvalues of $F_+$ are the symmetric eigenvalues of the equilateral triangle (the eigenvalues whose eigenfunctions are symmetric across the bisecting line). Therefore, after rescaling we see it suffices to show
\[
\frac{M^s_j(E_1)}{M_j(E_1)} \geq
\begin{cases}
11/6 , & \text{for $j=3,6$ and $j \geq 10$,} \\
8/5 , & \text{for $j=4,5,7,8,9$,}
\end{cases}
\]
where $E_1$ is an equilateral triangle with diameter $1$. For $j \leq 192$ this desired inequality follows by direct calculation of the eigenvalues (Lemma~\ref{explicit} in the Appendix).

Next, from Lemmas~\ref{counting1} and \ref{counting2} and an elementary estimate we find
\[
\frac{\mu^s_j}{\mu_j}>\frac{\frac{32\pi}{\sqrt{3}}(j-\frac{5}{4})-32\sqrt{2\sqrt{3}\pi(j-\frac{5}{4})+9}+96} {\frac{16\pi}{\sqrt{3}}(j+\frac{1}{2})-8\sqrt{\frac{4\pi}{\sqrt{3}}(j+\frac{1}{2})+1}+8}>\frac{11}{6}
\]
for all $j \geq 193$. Hence the inequality $M^s_j/M_j \geq 11/6$ extends from $j=192$ to all $j \geq 193$.

\section{Proof of Corollary~\ref{mu3} and Proposition~\ref{ellipse}} \label{misc_proof}

These results rely on a special kind of domain monotonicity holding for Neumann eigenvalues.
\begin{lemma}[Stretching] \label{stretching}
Let $\Omega$ be a Lipschitz domain in the plane. For $t > 1$, let
$\Omega_t = \{ (x,ty) : (x,y) \in \Omega \}$ be the domain obtained
by stretching $\Omega$ by the factor $t$ in the $y$ direction. Then
\[
\mu_j(\Omega_t) \leq \mu_j(\Omega) , \qquad j \geq 2 .
\]
If equality holds for some $j \geq 2$, then there exists a corresponding eigenfunction on $\Omega$ that depends only on $x$.
\end{lemma}
\begin{proof}[Proof of Lemma~\ref{stretching}]
The eigenvalue problem $-(v_{xx}+v_{yy}) = \mu v$ on $\Omega_t$ has Rayleigh quotient
\[
R[v] = \frac{\int_{\Omega_t} (v_x^2 + v_y^2) \, dxdy}{\int_{\Omega_t} v^2 \, dxdy} .
\]
We pull back to $\Omega$ by writing $u(x,y)=v(x,ty)$, so that $R[v]$ equals
\[
R_t[u] = \frac{\int_\Omega (u_x^2 + t^{-2} u_y^2) \, dxdy}{\int_\Omega u^2 \, dxdy} .
\]
This quotient is smaller for $t>1$ than for $t=1$, and so $\mu_j(\Omega_t) \leq \mu_j(\Omega)$ by the variational characterization of eigenvalues \cite[p.~97]{B80}.

We prove the equality statement for $j=2$, and leave the higher values of $j$ to the reader. Suppose $\mu_2(\Omega_t) = \mu_2(\Omega)$. Let $u$ be a second Neumann eigenfunction on $\Omega$. Then $u$ has mean value $0$ on $\Omega$, so that $v$ has mean value $0$ on $\Omega_t$. Hence $v$ is a valid trial function for $\mu_2(\Omega_2)$, and so
\[
\mu_2(\Omega_t) \leq R[v] = R_t[u]  \leq R[u] = \mu_2(\Omega) .
\]
Because equality holds in the second inequality, we conclude that $u_y \equiv 0$. That is, $u$ depends only on $x$.
\end{proof}

\subsection*{Proof of Corollary~\ref{mu3}}
Consider a non-equilateral triangle $T$. We may assume $T$ is subequilateral, for if not then it can be stretched in the direction perpendicular to its longest side, until one of the other two sides has the same length as the longest one; this subequilateral triangle has smaller $\mu_3$ than the original one, by Lemma~\ref{stretching}, and has the same diameter.

Among subequilateral triangles, $\mu_2D^2$ is maximal for the equilateral by a result of Laugesen and Siudeja \cite[Section 6]{LS10}:
\[
\mu_2D^2 \big|_T < \mu_2D^2 \big|_E .
\]
Furthermore, $(\mu_2+\mu_3)D^2$ is minimal for the equilateral by Theorem~\ref{main}:
\[
(\mu_2+\mu_3)D^2 \big|_T >  (\mu_2+\mu_3)D^2 \big|_E .
\]
Subtracting these two inequalities shows for subequilateral triangles that
\[
\mu_3D^2 \big|_T > \mu_3D^2 \big|_E .
\]

\subsection*{Proof of Proposition~\ref{ellipse}}
Each ellipse can be stretched to a circle of the same diameter. The Neumann eigenvalues strictly decrease under such stretching, by Lemma~\ref{stretching}.

\appendix{\section{Equilateral triangles, rectangles and their eigenvalues} \label{equilateral}

The frequencies of the equilateral triangle were derived
roughly 150 years ago by Lam\'{e} \cite[pp.~131--135]{L66}. For our Neumann situation, one can adapt the treatment of the Dirichlet case given by Mathews and Walker's text \cite[pp.~237--239]{MW70}, or in the paper by Pinsky \cite{Pi85}; or else see the exposition of the Neumann case by McCartin \cite{M02}.

The equilateral triangle $E_1$ with sidelength $1$ has Neumann eigenvalues forming a doubly-indexed sequence:
\[
\sigma_{m,n} = (m^2 + mn + n^2) \cdot \frac{16 \pi^2}{9} , \qquad m,n \geq 0 .
\]
For example,
\[
\mu_1 = 0 = \sigma_{0,0}, \qquad \qquad \mu_2 = \mu_3 =1 \cdot \frac{16 \pi^2}{9} = \sigma_{1,0}=\sigma_{0,1},
\]
\[
\mu_4 = 3 \cdot \frac{16 \pi^2}{9} = \sigma_{1,1}, \qquad \mu_5 = \mu_6 =4 \cdot \frac{16 \pi^2}{9} = \sigma_{2,0}=\sigma_{0,2}.
\]

Now consider a line of symmetry of $E_1$. Indices with $m > n$ correspond to eigenfunctions that are antisymmetric across that line (see McCartin \cite{M02}). Indices with $m \leq n$ correspond to symmetric eigenfunctions. Denote the corresponding ``symmetric eigenvalues'' by $0=\mu_1^s < \mu_2^s \leq \mu_3^s \leq \ldots$.


\captionsetup[subfigure]{parskip=0pt,margin=0pt}
\newcommand{\binomb}[2]{\genfrac{[}{]}{0pt}{}{#1}{#2}}
\newcommand{\frace}[2]{\genfrac{}{}{0pt}{}{#2}{#1}}
\newcommand\T{\rule{0pt}{3ex}}
\newcommand\B{\rule[-1.7ex]{0pt}{0pt}}

\begin{table}
  \centering
  \begin{tabular}{cccccccccc}
\toprule
\T\B$\!\!\frace{(0,0)}{[0,0]}\!\!$&$ \!\!\frace{(0,1)}{[0,1]}\!\!$&$
\!\!\frace{(1,0)}{[1,1]}\!\!$&$\!\!\frace{(1,1)}{[0,2]}\!\!$&$ \!\!\frace{(0,2)}{[1,2]}\!\!$&$
\!\!\frace{(2,0)}{[0,3]}\!\!$&$ \!\!\frace{(1, 2)}{[2,2]}\!\!$&$ \!\!\frace{(2,1)}{[1,3]}\!\!$&$ \!\!\frace{(0,3)}{[0,4]}\!\!$&$ \!\!\frace{(3,0)}{[2,3]}\!\!$\\
\T\B$\!\!\frace{(2,2)}{[1,4]}\!\!$&$ \!\!\frace{(1,3)}{[0,5]}\!\!$&$ \!\!\frace{(3,1)}{[3,3]}\!\!$&$\!\!\frace{(0,4)}{[2,4]}\!\!$&$ \!\!\frace{(4,0)}{[1,5]}\!\!$&$ \!\!\frace{(2,3)}{[0,6]}\!\!$&$ \!\!\frace{(3,2)}{[3,4]}\!\!$&$ \!\!\frace{(1,4)}{[2,5]}\!\!$&$ \!\!\frace{(4,1)}{[1,6]}\!\!$&$ \!\!\frace{(0,5)}{[4,4]}\!\!$\\
\T\B$\!\!\frace{(5,0)}{[0,7]}\!\!$&$ \!\!\frace{(3,3)}{[3,5]}\!\!$&$ \!\!\frace{(2,4)}{[2,6]}\!\!$&$ \!\!\frace{(4,2)}{[1,
    7]}\!\!$&$ \!\!\frace{(1,5)}{[4,5]}\!\!$&$ \!\!\frace{(5,1)}{[3,6]}\!\!$&$ \!\!\frace{(0,6)}{[0,8]}\!\!$&$ \!\!\frace{(6,0)}{[2,7]}\!\!$&$ \!\!\frace{(3,4)}{[1,8]}\!\!$&$ \!\!\frace{(4,3)}{[5,5]}\!\!$\\
\T\B$\!\!\frace{(2,5)}{[4,6]}\!\!$&$ \!\!\frace{(5,2)}{[3,7]}\!\!$&$ \!\!\frace{(1,6)}{[0,9]}\!\!$&$ \!\!\frace{(6,1)}{[2,8]}\!\!$&$ \!\!\frace{(4,4)}{[1,9]}\!\!$&$ \!\!\frace{(0,7)}{[5,6]}\!\!$&$ \!\!\frace{(3,5)}{[4,7]}\!\!$&$ \!\!\frace{(5,3)}{[3,8]}\!\!$&$ \!\!\frace{(7,0)}{[0,10]}\!\!$&$ \!\!\frace{(2,6)}{[2,9]}\!\!$\\
\T\B$\!\!\frace{(6,2)}{[6,6]}\!\!$&$ \!\!\frace{(1,7)}{[5,7]}\!\!$&$ \!\!\frace{(7,1)}{[1,10]}\!\!$&$ \!\!\frace{(4,5)}{[4,8]}\!\!$&$ \!\!\frace{(5,4)}{[3,9]}\!\!$&$ \!\!\frace{(3,6)}{[0,11]}\!\!$&$ \!\!\frace{(6,3)}{[2,10]}\!\!$&$ \!\!\frace{(0,8)}{[6,7]}\!\!$&$ \!\!\frace{(8,0)}{[5,8]}\!\!$&$ \!\!\frace{(2,7)}{[1,11]}\!\!$\\
\T\B$\!\!\frace{(7,2)}{[4,9]}\!\!$&$ \!\!\frace{(1,8)}{[3,10]}\!\!$&$ \!\!\frace{(8,1)}{[0,12]}\!\!$&$ \!\!\frace{(5,5)}{[2,11]}\!\!$&$ \!\!\frace{(4,6)}{[7,7]}\!\!$&$ \!\!\frace{(6,4)}{[6,8]}\!\!$&$ \!\!\frace{(3,7)}{[5,9]}\!\!$&$ \!\!\frace{(7,3)}{[4,10]}\!\!$&$ \!\!\frace{(0,9)}{[1,12]}\!\!$&$ \!\!\frace{(9,0)}{[3,11]}\!\!$\\
\T\B$\!\!\frace{(2,8)}{[0,13]}\!\!$&$ \!\!\frace{(8,2)}{[7,8]}\!\!$&$ \!\!\frace{(1,9)}{[6,9]}\!\!$&$ \!\!\frace{(5,6)}{[2,12]}\!\!$&$ \!\!\frace{(6,5)}{[5,10]}\!\!$&$ \!\!\frace{(9,1)}{[4,11]}\!\!$&$ \!\!\frace{(4,7)}{[1,13]}\!\!$&$ \!\!\frace{(7,4)}{[3,12]}\!\!$&$ \!\!\frace{(3,8)}{[8,8]}\!\!$&$ \!\!\frace{(8,3)}{[7,9]}\!\!$\\
\T\B$\!\!\frace{(0,10)}{[0,14]}\!\!$&$ \!\!\frace{(10,0)}{[6,10]}\!\!$&$ \!\!\frace{(2,9)}{[2,13]}\!\!$&$ \!\!\frace{(9,2)}{[5,11]}\!\!$&$ \!\!\frace{(6,6)}{[4,12]}\!\!$&$ \!\!\frace{(5,7)}{[1,14]}\!\!$&$ \!\!\frace{(7,5)}{[3,13]}\!\!$&$ \!\!\frace{(1,10)}{[8,9]}\!\!$&$ \!\!\frace{(10,1)}{[7,10]}\!\!$&$ \!\!\frace{(4,8)}{[6,11]}\!\!$\\
\T\B$\!\!\frace{(8,4)}{[0,15]}\!\!$&$ \!\!\frace{(3,9)}{[2,14]}\!\!$&$ \!\!\frace{(9,3)}{[5,12]}\!\!$&$ \!\!\frace{(0,11)}{[4,13]}\!\!$&$ \!\!\frace{(11,0)}{[1,15]}\!\!$&$ \!\!\frace{(2,10)}{[9,9]}\!\!$&$ \!\!\frace{(10,2)}{[8,10]}\!\!$&$ \!\!\frace{(6,7)}{[3,14]}\!\!$&$ \!\!\frace{(7,6)}{[7,11]}\!\!$&$ \!\!\frace{(5,8)}{[6,12]}\!\!$\\
\T\B$\!\!\frace{(8,5)}{[0,16]}\!\!$&$ \!\!\frace{(1,11)}{[2,15]}\!\!$&$ \!\!\frace{(4,9)}{[5,13]}\!\!$&$ \!\!\frace{(9,4)}{[4,14]}\!\!$&$ \!\!\frace{(11,1)}{[9,10]}\!\!$&$ \!\!\frace{(3,10)}{[1,16]}\!\!$&$ \!\!\frace{(10,3)}{[8,11]}\!\!$&$ \!\!\frace{(0,12)}{[7,12]}\!\!$&$ \!\!\frace{(12,0)}{[3,15]}\!\!$&$ \!\!\frace{(2,11)}{[6,13]}\!\!$\\
\T\B$\!\!\frace{(7,7)}{[0,17]}\!\!$&$ \!\!\frace{(11,2)}{[5,14]}\!\!$&$ \!\!\frace{(6,8)}{[2,16]}\!\!$&$ \!\!\frace{(8,6)}{[10,10]}\!\!$&$ \!\!\frace{(5,9)}{[4,15]}\!\!$&$ \!\!\frace{(9,5)}{[9,11]}\!\!$&$ \!\!\frace{(4,10)}{[8,12]}\!\!$&$ \!\!\frace{(10,4)}{[1,17]}\!\!$&$ \!\!\frace{(1,12)}{[7,13]}\!\!$&$ \!\!\frace{(12,1)}{[3,16]}\!\!$\\
\T\B$\!\!\frace{(3, 11)}{[6,14]}\!\!$&$ \!\!\frace{(11, 3)}{[0,18]}\!\!$&$ \!\!\frace{(0,13)}{[5,15]}\!\!$&$ \!\!\frace{(7,8)}{[2,17]}\!\!$&$ \!\!\frace{(8,7)}{[10,11]}\!\!$&$ \!\!\frace{(13,0)}{[9,12]}\!\!$&$ \!\!\frace{(6,9)}{[4,16]}\!\!$&$ \!\!\frace{(9, 6)}{[8, 13]}\!\!$&$ \!\!\frace{(2, 12)}{[1,18]}\!\!$&$ \!\!\frace{(12,2)}{[7,14]}\!\!$\\
\T\B$\!\!\frace{(5,10)}{[3,17]}\!\!$&$ \!\!\frace{(10,5)}{[6,15]}\!\!$&$ \!\!\frace{(4,11)}{[0,19]}\!\!$&$ \!\!\frace{(11,4)}{[5,16]}\!\!$&$ \!\!\frace{(1,13)}{[11,11]}\!\!$&$ \!\!\frace{(13,1)}{[2,18]}\!\!$&$ \!\!\frace{(3, 12)}{[10,12]}\!\!$&$ \!\!\frace{(12,3)}{[9,13]}\!\!$&$ \!\!\frace{(8,8)}{[8,14]}\!\!$&$ \!\!\frace{(7,9)}{[4,17]}\!\!$\\
\T\B$\!\!\frace{(9, 7)}{[7, 15]}\!\!$&$ \!\!\frace{(0, 14)}{[1, 19]}\!\!$&$ \!\!\frace{(6, 10)}{[3, 18]}\!\!$&$ \!\!\frace{(10,
    6)}{[6, 16]}\!\!$&$ \!\!\frace{(14, 0)}{[11, 12]}\!\!$&$ \!\!\frace{(2, 13)}{[5, 17]}\!\!$&$ \!\!\frace{(13,
    2)}{[10, 13]}\!\!$&$ \!\!\frace{(5, 11)}{[0, 20]}\!\!$&$ \!\!\frace{(11, 5)}{[2, 19]}\!\!$&$ \!\!\frace{(4,
    12)}{[9, 14]}\!\!$\\
    \T\B$\!\!\frace{(12,4)}{[8, 15]}\!\!$&$ \!\!\frace{(1, 14)}{[4, 18]}\!\!$&$ \!\!\frace{(14,1)}{[7, 16]}\!\!$&$ \!\!\frace{(3,
    13)}{[1, 20]}\!\!$&$ \!\!\frace{(8, 9)}{[3,19]}\!\!$&$ \!\!\frace{(9, 8)}{[6, 17]}\!\!$&$ \!\!\frace{(13,3)}{[12, 12]}\!\!$&$ \!\!\frace{(7, 10)}{[11, 13]}\!\!$&$ \!\!\frace{(10,7)}{[10, 14]}\!\!$&$ \!\!\frace{(6,
    11)}{[5, 18]}\!\!$\\
\T\B$\!\!\frace{(11, 6)}{[0, 21]}\!\!$&$ \!\!\frace{(0, 15)}{[9, 15]}\!\!$&$ \!\!\frace{(15, 0)}{[2, 20]}\!\!$&$ \!\!\frace{(2,
    14)}{[8, 16]}\!\!$&$ \!\!\frace{(14, 2)}{[4, 19]}\!\!$&$ \!\!\frace{(5, 12)}{[7, 17]}\!\!$&$ \!\!\frace{(12,
    5)}{[1, 21]}\!\!$&$ \!\!\frace{(4, 13)}{[6, 18]}\!\!$&$ \!\!\frace{(13, 4)}{[3, 20]}\!\!$&$ \!\!\frace{(1,
    15)}{[12, 13]}\!\!$\\
\T\B$\!\!\frace{(15, 1)}{[11, 14]}\!\!$&$ \!\!\frace{(9, 9)}{[10, 15]}\!\!$&$ \!\!\frace{(8, 10)}{[5, 19]}\!\!$&$ \!\!\frace{(10,8)}{[9,16]}\!\!$&$ \!\!\frace{(3, 14)}{[0, 22]}\!\!$&$ \!\!\frace{(7,11)}{[2,21]}\!\!$&$ \!\!\frace{(11,7)}{[8, 17]}\!\!$&$ \!\!\frace{(14, 3)}{[4, 20]}\!\!$&$ \!\!\frace{(6, 12)}{[7, 18]}\!\!$&$ \!\!\frace{(12,
    6)}{[1, 22]}\!\!$\\
    \T\B$\!\!\frace{(0, 16)}{[13, 13]}\!\!$&$ \!\!\frace{(16,0)}{[12,14]}\!\!$&$ \!\!\frace{(2, 15)}{[6, 19]}\!\!$&$ \!\!\frace{(5,
    13)}{[11, 15]}\!\!$&$ \!\!\frace{(13, 5)}{[3, 21]}\!\!$&$ \!\!\frace{(15, 2)}{[10, 16]}\!\!$&$ \!\!\frace{(4,
    14)}{[9, 17]}\!\!$&$ \!\!\frace{(14, 4)}{[5, 20]}\!\!$&$ \!\!\frace{(9, 10)}{[0, 23]}\!\!$&$ \!\!\frace{(10,
    9)}{[2, 22]}\!\!$\\
    \T\B$\!\!\frace{(1, 16)}{[8, 18]}\!\!$&$ \!\!\frace{(8, 11)}{[4, 21]}\!\!$&$ \!\!\frace{(11, 8)}{[7, 19]}\!\!$&$ \!\!\frace{(16,
    1)}{[13, 14]}\!\!$&$ \!\!\frace{(7, 12)}{[12, 15]}\!\!$&$ \!\!\frace{(12, 7)}{[1, 23]}\!\!$&$ \!\!\frace{(3,
    15)}{[11, 16]}\!\!$&$ \!\!\frace{(15, 3)}{[6, 20]}\!\!$&$ \!\!\frace{(6, 13)}{[3, 22]}\!\!$&$ \!\!\frace{(13,
    16)}{[10, 17]}\!\!$\\
    \T\B$\!\!\frace{(0, 17)}{[9, 18]}\!\!$&$ \!\!\frace{(17, 0)}{[5, 21]}\!\!$&$ \!\!\frace{(5, 14)}{[0, 24]}\!\!$&$ \!\!\frace{(14,
    5)}{[8, 19]}\!\!$&$ \!\!\frace{(2, 16)}{[2, 23]}\!\!$&$ \!\!\frace{(16, 2)}{[4, 22]}\!\!$&$ \!\!\frace{(10,
    10)}{[14, 14]}\!\!$&$ \!\!\frace{(4, 15)}{[7, 20]}\!\!$&$ \!\!\frace{(9, 11)}{[13, 15]}\!\!$&$ \!\!\frace{(11,
    9)}{[12, 16]}\!\!$\\
\bottomrule
 \end{tabular}
 \medskip
  \caption{Pairs of integers $(m,n)$ giving the first $200$ eigenvalues $\mu_j$ along with pairs $[m,n]$ giving the first $200$ symmetric eigenvalues $\mu_j^s$, for an equilateral triangle. The index $j$ increases from $1$ to $10$ across the first row, and so on.}
  \label{tabledata}
\end{table}

\begin{lemma} \label{explicit}
For $j=3,6$, and for $10 \leq j \leq 200$, we have
\begin{equation} \label{appeq1}
(\mu_2^s+\cdots+\mu_j^s) > \frac{11}{6} (\mu_2+\cdots+\mu_j) .
\end{equation}
For $j=4,5,7,8,9$, we have a weaker inequality,
\begin{equation} \label{appeq2}
(\mu_2^s+\cdots+\mu_j^s) \geq \frac{8}{5} (\mu_2+\cdots+\mu_j) ,
\end{equation}
with equality for $j=4$ and strict inequality for $j=5,7,8,9$.
\end{lemma}
\begin{proof}[Proof of Lemma~\ref{explicit}]
Begin by computing the first $200$ eigenvalues $\mu_j$ and symmetric eigenvalues $\mu_j^s$, using the indices $m$ and $n$ listed in Table~\ref{tabledata}. Estimates \eqref{appeq1} and \eqref{appeq2} can then easily be checked. As a shortcut for \eqref{appeq1}, one can verify that $\mu_j^s > \frac{11}{6} \mu_j$ whenever $28 \leq j \leq 200$, so that \eqref{appeq1} holds for $28 \leq j \leq 200$ as soon as the case $j=27$ has been checked.
\end{proof}

\section*{Acknowledgments}
We acknowledge support from the National Science Foundation grant
DMS 08-38434 ``EMSW21-MCTP: Research Experience for Graduate Students''. We also thank Bart{\l}omiej Siudeja for suggesting that we investigate Neumann eigenvalues.

\end{document}